\documentclass[reqno,preprint,12pt]{amsart}
\usepackage{enumerate}

\usepackage[colorlinks=true]{hyperref}
\hypersetup{urlcolor=blue, linkcolor=blue, citecolor=red}

\usepackage[latin9]{inputenc} 
\usepackage[english]{babel}

\newtheorem{theo}{Theorem}[section]

\theoremstyle{definition}

\newcommand\MScN[1]{\href{http://www.ams.org/mathscinet-getitem?mr=#1}{\nolinkurl{(#1)}}}

\newcommand\DOI[1]{\href{http://dx.doi.org/#1}{(doi: \nolinkurl{#1})}}

\newcommand\LINK[1]{\href{#1}{(link: \nolinkurl{#1})}}

\title[Stokes problem with Navier boundary condition]{Maximal $L^p$-$L^q$ regularity for the Stokes problem with Navier-type boundary conditions}
\author{Hind Al Baba}
\address{Institute of Mathematics of the Czech Academy of Sciences, Zitna 25, 115 67   Praha 1 
Czech Republic, and  Laboratoire de Math\'ematiques et de leurs applications-Pau, UMR, CNRS 5142, Batiment IPRA, Universit\'e de Pau et des pays de L'Adour, Avenue de L'universit\'e, Bureau 012, BP 1155, 64013 Pau cedex, France}
\email{albaba@math.cas.cz}
\email{hind.albaba@univ-pau.fr}

\author{Ch\'erif Amrouche}
\address{Laboratoire de Math\'ematiques et de leurs applications-Pau, UMR, CNRS 5142, Batiment IPRA, Universit\'e de Pau et des pays de L'Adour, Avenue de L'universit\'e, Bureau 225, BP 1155, 64013 Pau cedex, France}
\email{cherif.amrouche@univ-pau.fr}

\author{Miguel Escobedo}
\address{Departamento de Matem\'aticas
Facultad de Ciencias y Tecnolog\'ia
Universidad del Pa\'is Vasco
(UPV/EH), Apartado 644, Bilbao 48080, Spain}
\email{ miguel.escobedo@ehu.es}
\thanks{}
\begin{document}
\maketitle

\vspace{-0.5cm}
\begin{center}
\today
\end{center}

\begin{abstract} Maximal $L^p$-$L^q$ regularity is proved for the strong, weak and very weak  solutions of the inhomogeneous Stokes problem with Navier-type boundary conditions in a bounded domain $\Omega$, not necessarily simply connected.  This extends previous results of the authors (2017).
\end{abstract}

\vspace{0.25cm}


\noindent MSC[2008] 35B65, 35D30, 35D35, 35K20, 35Q30,  76D05,  76D07, 76N10.

\section{Introduction}\label{Introduction}
We consider in this paper the maximal $L^p$-$L^q$ regularity for the following Stokes problem with  slip frictionless boundary conditions involving the tangential component of the velocity vortex instead of the stress tensor :
\begin{eqnarray}
\frac{\partial\textbf{u}}{\partial t} - \Delta \textbf{u
}+\nabla\pi=\textbf{f}, & \mathrm{div}\,\textbf{u}=0 \qquad\qquad\textrm{in}\,\,\,
\Omega\times (0,T),\label{StokesProblem} \\
\textbf{u}\cdot\textbf{n}=0, & \mathrm{\textbf{curl}}\,\textbf{u}\times \textbf{n} = \textbf{0}\qquad\textrm{on}\,\,\,
\Gamma\times (0,T),\label{Navierbc}\\
\textbf{u}(0)= \textbf{0}& \qquad \qquad\textrm{in} \,\,\,
\Omega, \label{u0}
\end{eqnarray} 
where $\Omega$ is a bounded domain of $\mathbb{R}^{3}$ of class $C^{2,1}$ not necessarily simply-connected, $\Gamma $ is its boundary and $\textbf{n}$ is the exterior unit normal vector on $\Gamma $. The unknowns  $\textbf{u}$ and $\pi$ denote respectively the velocity field and the pressure of a fluid occupying the domain $\Omega$, while $\textbf{u}(0)$ and $\textbf{f}$ represent respectively the given initial velocity and the external force.  Under these conditions, the Stokes problem has a non trivial kernel  $ \textbf{K}_{\tau}(\Omega)$ (see (\ref{noinailb}) below).

Given a  Cauchy-Problem of the form:
\begin{equation}\label{abstcauchprob}
\left\{
\begin{array}{cc}
\frac{\partial u}{\partial t}\,+\,\mathcal{A}\,u(t)=\textit{f}(t)& 0\leq t\leq T\\
u(0)=0,&
\end{array}
\right.
\end{equation}
where $-\mathcal{A}$ is the infinitesimal generator of a semi-group $e^{-t\mathcal A}$ on a Banach space $X$  and $\textit{f}\in L^{p}(0,T;\,X)$,
we say that a solution $u$ satisfies  the maximal $L^p$-$L^q$ regularity if
\begin{equation}\label{regulacp}
u\in W^{1,p}(0,T;\,X)\cap L^{p}(0,T;\,D(\mathcal{A})).
\end{equation}
It is known that the analyticity of $e^{-t\mathcal A}$ is not enough  to ensure that property to be satisfied, although it is enough when  $X$ is a Hilbert space  (cf. \cite{DV}, \cite{CL}).

When $1<p, q<\infty$, the  maximal $L^p$-$L^q$ regularity  has been proved by the authors in \cite{ARMA} for solutions to (\ref{StokesProblem})-(\ref{u0})  lying in the orthogonal of the kernel $ \textbf{K}_{\tau}(\Omega)$. In terms of the abstract example (\ref{abstcauchprob}), the main argument of the proof, based on the use of  the results of \cite{Geissert}, was to show that the  pure imaginary powers of  $(I+\mathcal A)$ are suitably bounded operators, and deduce that  so where the imaginary powers of  $\mathcal A$. That  could only  be done assuming the operator $\mathcal A$ to be invertible, but that is not the case of the Stokes operator on a non simply-connected domain, with boundary conditions  (\ref{Navierbc}). The maximal regularity result was then proved only for  the restriction of the Stokes operator to the kernel's orthogonal, where it was of course invertible. The purpose of the present work is to extend that result to the solutions of (\ref{StokesProblem})-(\ref{u0}) that do not  necessarily lie in the orthogonal of $ \textbf{K}_{\tau}(\Omega)$. The idea is to decompose the solution as an element of the kernel and an element of its orthogonal and to apply the result of  \cite{ARMA}.

We are  interested in three different types of solutions for   (\ref{StokesProblem})-(\ref{u0}). The first, that we call strong solutions, are solutions $\textbf{u}$ that belong to $L^p(0, T, \textbf{L}^q(\Omega))$ type spaces. The second, called weak solutions, are solutions (in a suitable sense)  $\textbf{u}(t)$ that may be writen for a.e. $t>0$, as $\textbf{u}(t)=\textbf{v}(t)+\nabla w(t)$ where 
$\textbf{v}(t)\in L^p(0, T;  \textbf{L}^{q}(\Omega)))$ and $w\in L^p(0, T; L^{q}(\Omega) )$. The third and last, called very weak, are  solutions $\textbf{u}(t)$ that may be decomposed as before but where now $w\in L^p(0, T; W^{-1,\, q}(\Omega) )$ (cf \cite{ARMA} for more details). Of course, these different types of solutions  correspond to data $\textbf{u}(0)$ and $\textbf{f}$ with different regularity properties.

There is a wide literature on the maximal regularity for the Stokes problem with different type of boundary conditions and different domains. Among the firsts articles on this problem
we may mention  \cite{Solonnikov} by  V. A. Solonnikov.  The works by Y. Giga and H. Sohr  \cite{GiGa3, GiGa4} consider that question for the Stokes problem with Dirichlet boundary conditions in bounded and unbounded domains; J. Saal \cite{Saal} for the Stokes problem with homogeneous Robin boundary conditions in the half space $\mathbb{R}^{3}_{+}$; R. Shimada \cite{Shimada} for the Stokes problem with non-homogeneous Robin boundary conditions. The maximal regularity  for general parabolic problems is treated in detail  in the long report  \cite{KW} by P. C. Kunstmann and L. Weisand.

In the next Section we introduce some  notations and recall several results, already known, that are needed thereafter. Our main results are stated, and their proofs given in Section 3. 

\section{Stokes operator}
\label{section2}
In order to obtain  strong, weak and very weak solutions to our problem (\ref{StokesProblem})-(\ref{u0}),  we introduced in \cite{ARMA}  three different extensions $A_{p}$, $B_p$, $C_p$,  of the Stokes operators with boundary conditions (\ref{Navierbc}), defined in different spaces of distributions with different regularity properties. Throughout this paper, if not stated otherwise, $p$ will be a real number such that $1<p<\infty$.

We first consider $A_p$,  the Stokes operator with the boundary conditions (\ref{Navierbc}) on  the space $\textbf{L}^{p}_{\sigma,\tau}(\Omega)$ given by 
\begin{equation}\label{Lpsigmatau}
\textbf{L}^{p}_{\sigma,\tau}(\Omega)=\Big\{\textbf{f}\in\textbf{L}^{p}(\Omega);\,\,\,\mathrm{div}\textbf{f}=0\,\,\mathrm{in}\,\,\Omega,\,\, \textbf{f}\cdot\textbf{n}=0\,\,\mathrm{on}\,\,\Gamma\Big\}.
\end{equation} 
By \cite[Corollary 3.7]{ARMA}, this is a well defined subspace of 
\begin{equation}
\label{Hpdiv}
\textbf{H}^{p}(\textrm{div},\Omega)\,=\,\big\{\textbf{v}\in\textbf{L}^{p}(\Omega);\,\,\,\textrm{div}\,\textbf{v}\in\textbf{L}^{p}(\Omega)\big\},
\end{equation}
equipped with the graph norm. 
As described in \cite[Section 3]{conm}, $A_p$  is a closed linear densely defined operator on $\textbf{L}^{p}_{\sigma,\tau}(\Omega)$ defined as follows
$$D(A_{p})= \Big\{ \textbf{u} \in
\textbf{W}^{2,p}(\Omega); \mathrm{div}\,\textbf{u}=
0 \,\mathrm{in}\,\,\Omega,
\textbf{u}\cdot\textbf{n}=0, \mathbf{curl}\,\textbf{u}\times\textbf{n}=\textbf{0}
\,\,\mathrm{on}\,\,\Gamma \Big\}
$$
\begin{equation}\label{defistokesoperator}
\forall\,\textbf{u}\in D(A_{p}),\quad A_{p}\textbf{u}=-P\Delta\textbf{u}\quad\mathrm{in}\,\,\Omega.
\end{equation}
The operator $P$ in (\ref{defistokesoperator}), is the Helmholtz projection defined as follows:
\begin{equation}\label{helmholtzproj}
P:\textbf{L}^{p}(\Omega)\longmapsto\textbf{L}^{p}_{\sigma,\tau}(\Omega);\,\,\,\,\,\,\forall\,\textbf{f}\in\textbf{L}^{p}(\Omega):\quad P\,\textbf{f}\,=\,\textbf{f}\,-\,\mathrm{\textbf{grad}}\,\pi,
\end{equation}
where $\pi\in W^{1,p}(\Omega)/\mathbb{R}$ is the unique solution of the following weak Neuman Problem (cf. \cite{Si}):
\begin{equation}\label{wn.1}
\mathrm{div}\,(\mathrm{\textbf{grad}}\,\pi\,-\,\textbf{f})=0\,\,\,
\mathrm{in}\,\,\Omega,\,\,\,\,\,
(\mathrm{\textbf{grad}}\,\pi\,-\,\textbf{f})\cdot\textbf{n}=0,\,\,\,\mathrm{on}\,\,\Gamma.
\end{equation} 

It is known that, due to the slipping frictionless boundary condition (\ref{Navierbc}), the pressure gradient disappears  in the Stokes operator (cf. \cite[Proposition 3.1]{conm}). As a result the Stokes problem with the boundary condition (\ref{Navierbc}) is reduced to the study of a vectorial Laplace like problem under a free-divergence condition and the boundary conditions (\ref{Navierbc}). 
$$\forall\,\,\textbf{u}\in D (A_{p}),\qquad
A_{p}\textbf{u}\,=\,-\Delta\textbf{u}\quad\mathrm{in}\,\,\Omega.$$

 We also recall that the operator $-A_{p}$ is sectorial and generates a bounded analytic semi-group on $\textbf{L}^{p}_{\sigma,\tau}(\Omega),\,$ for all $1<p<\infty$ (cf. \cite[Theorem 4.12]{conm}). We denote by $e^{-t A_{p}}$ the analytic semi-group associated to the operator $A_{p}$ in $\textbf{L}^{p}_{\sigma,\tau}(\Omega)$.

\medskip

When $\Omega $ is not simply-connected, the Stokes  operator with boundary condition (\ref{Navierbc}) has a non trivial kernel included in all the $L^p$ spaces for $p\in (1, \infty)$. 
It may be characterized as follows (see \cite{Cherif-Nour-M3AS})
\begin{equation}\label{noinailb}
 \textbf{K}_{\tau}(\Omega)\,=\,\big\{\textbf{v}\in\textbf{L}^{p}_{\sigma,\tau}(\Omega);\,\,\mathrm{div}\,\textbf{v}=0,\,\,\mathrm{\textbf{curl}}\,\textbf{v}=\textbf{0}\,\,\textrm{in}\,\,\Omega\big\}.
 \end{equation}
The restriction of the Stokes operator $A_p$ to the subspace
\begin{equation}\label{Xp}
\textbf{X}_{p}\,=\,\big\{\textbf{f}\in\textbf{L}^{p}_{\sigma,\tau}(\Omega);\,\,\,\int_{\Omega}\textbf{f}\cdot\overline{\textbf{v}}\,\textrm{d}\,x=0,\,\,\forall\,\,\textbf{v}\in\textbf{K}_{\tau}(\Omega)\big\},
 \end{equation}
 gives a sectorial operator which is invertible, with bounded inverse. 
 Notice that 
\begin{equation}
\label{S2E10}
\textbf{L}^{p}_{\sigma,\tau}(\Omega)=\textbf{K}_{\tau}(\Omega)\oplus\textbf{X}_{p}.
\end{equation}
We consider now the extension of $A_p$ to the following subspace of $[\textbf{H}^{p}_{0}(\textrm{div},\Omega)]'$ (the dual space of $\textbf{H}^{p}_{0}(\textrm{div},\Omega)$):
\begin{equation}
[\textbf{H}^{p'}_{0}(\mathrm{div},\Omega)]'_{\sigma, \tau}=\{\textbf{f}\in[\textbf{H}^{p'}_{0}(\mathrm{div},\Omega)]';\,\mathrm{div}\,\textbf{f}=0\,\mathrm{in}\,\Omega,\, \textbf{f}\cdot \textbf{n}=0\,\, \mathrm{on}\,\,\Gamma\}.
\end{equation}
By \cite[Corollary 3.7]{ARMA}, that space is well defined, and the extended operator, denoted $B_p$, is a closed linear densely defined operator  such as:
$$D(B_{p})\subset[\textbf{H}^{p'}_{0}(\mathrm{div},\Omega)]'_{\sigma, \tau}\longmapsto[\textbf{H}^{p'}_{0}(\mathrm{div},\Omega)]'_{\sigma, \tau},$$ 
$$
D(B_{p})=\big\{
\textbf{u}\in\textbf{W}^{1,p}(\Omega);\mathrm{div}\,\textbf{u}=0\,\,\textrm{in}\,\Omega,
\textbf{u}\cdot\textbf{n}\,=\,0,\mathbf{curl}\,\textbf{u}\times\textbf{n}=\textbf{0}\,\,\textrm{on}\,\,\Gamma\big\}
$$
and 
\begin{equation}\label{operatorB}
\forall\,\,\textbf{u}\in D (B_{p}),\quad
B_{p}\textbf{u}\,=\,-\Delta\textbf{u}\quad\mathrm{in}\,\,\Omega.
\end{equation}
By \cite[Corollary 4.2]{Amrouche1}, the domain $D(B_p)$ is well defined and, by  \cite[Theorem 4.15]{ARMA}  the operator $-B_{p}$ generates a bounded analytic semi-group on $\,[\textbf{H}^{p'}_{0}(\mathrm{div},\Omega)]'_{\sigma, \tau},\,$ for all $\,1<p<\infty$, whose restriction to 
 \begin{equation}\label{Yp}
\textbf{Y}_{p}\,=\,\left\lbrace\textbf{f}\in[\textbf{H}^{p'}_{0}(\mathrm{div},\Omega)]'_{\sigma,\tau};\,\,\forall\,\textbf{v}\in\textbf{K}_{\tau}(\Omega),\,\,\,\langle\textbf{f},\,\textbf{v}\rangle_{\Omega}\,=\,0  \right\rbrace, 
\end{equation} 
where $\langle .\,,\,.\rangle_{\Omega}=\langle .\,,\,.\rangle_{[\textbf{H}^{p'}_{0}(\mathrm{div},\Omega)]'\times\textbf{H}^{p'}_{0}(\mathrm{div},\Omega)}$, is a sectorial operator, invertible with bounded inverse. Notice also that:
\begin{equation}
\label{S2E11}
[\textbf{H}^{p'}_{0}(\mathrm{div},\Omega)]'_{\sigma,\tau}=\textbf{K}_{\tau}(\Omega)\oplus\textbf{Y}_{p}.
\end{equation}
 
In order to introduce our third  operator we first need the following space: 
\begin{equation}\label{tp}
\textbf{T}^{p}(\Omega)=\big\{\textbf{v}\in\textbf{H}^{p}_{0}(\mathrm{div},\Omega);\,\,\,\mathrm{div}\,\textbf{v}\in
{W}^{1,p}_{0}(\Omega)\big\}
\end{equation}
and  consider the following subspace
$$
[\textbf{T}^{p'}(\Omega)]'_{\sigma, \tau}=\{\textbf{f}\in(\textbf{T}^{p'}(\Omega))';\,\,\mathrm{div}\,\textbf{f}=0\quad \mathrm{in}\,\, \Omega \quad \mathrm{and}\quad \textbf{f}\cdot \textbf{n}=0\quad \mathrm{on}\,\,\Gamma\},
$$
that is well defined by \cite[Corollary 3.12]{ARMA}.

The Stokes operator $A_p$ can be extended to the space $[\textbf{T}^{p'}(\Omega)]'_{\sigma, \tau}$ (cf. \cite[Section 3.2.3]{ARMA}). This extension is a densely defined closed linear operator, denoted  $C_p$:
$$D(C_{p})\subset [\textbf{T}^{p'}(\Omega)]'_{\sigma, \tau}\longmapsto[\textbf{T}^{p'}(\Omega)]'_{\sigma, \tau},\quad\mathrm{where}$$ 
$$
D(C_{p})=\big\{
\textbf{u}\in\textbf{L}^{p}(\Omega);\,\mathrm{div}\,\textbf{u}=0\,\,\textrm{in}\,\,\Omega,\,\,
\textbf{u}\cdot\textbf{n}\,=\,0,\,\,\mathbf{curl}\,\textbf{u}\times\textbf{n}=\textbf{0}\,\,\textrm{on}\,\,\Gamma\big\}
$$
 and for all $\,\textbf{u}\in D (C_{p}),\,\,
C_{p}\textbf{u}\,=\,-\Delta\textbf{u}$ in $\Omega.$ The domain $D(C_p)$ is well defined by  \cite[Lemma 4.14]{Amrouche1}. The operator $-C_{p}$ generates a bounded analytic semi-group on $[\textbf{T}^{p'}(\Omega)]'_{\sigma, \tau}$ for all $1<p<\infty$ (see \cite[Theorem 4.18]{ARMA}). If we define now
\begin{equation}\label{Zp}
\textbf{Z}_{p}\,=\,\left\lbrace\textbf{f}\in[\textbf{T}^{p'}(\Omega)]'_{\sigma,\tau};\,\,\forall\,\textbf{v}\in\textbf{K}_{\tau}(\Omega),\,\,\,\langle\textbf{f},\,\textbf{v}\rangle_{\Omega}\,=\,0  \right\rbrace, 
\end{equation} 
where $\langle .\,,\,.\rangle_{\Omega}=\langle .\,,\,.\rangle_{[\textbf{T}^{p'}(\Omega)]'\times\textbf{T}^{p'}(\Omega)}$, then
\begin{equation}
\label{S2E14}[\textbf{T}^{p'}(\Omega)]'_{\sigma,\tau}=\textbf{K}_{\tau}(\Omega)\oplus\textbf{Z}_{p}
\end{equation}
and the restriction of the Stokes
operator to the space $\textbf{Z}_{p}$,  gives a sectorial operator, invertible with bounded inverse. 

\section{Maximal Regularity: our main results.}\label{Applications}

We consider  in this Section the  problem (\ref{StokesProblem})--(\ref{u0}) under different conditions of the external force $f$.  In our first result we assume $\textbf{f}\in
 L^{q}(0,T;\,\textbf{L}^{p}_{\sigma,\tau}(\Omega))$ and $1<p,q<\infty$. 
 
 \begin{theo} \label{existinhensp} 
 Let $1<p,q<\infty$ and $0<T\leq\infty$. Then for every $\textbf{f}\in L^{q}(0,T;\,\textbf{L}^{p}_{\sigma,\tau}(\Omega))$ there exists a unique solution $\textbf{u}$ of  (\ref{StokesProblem})--(\ref{u0}) satisfying
\begin{equation}\label{reglplqlap1}
\textbf{u}\in L^{q}(0,T_{0};\textbf{W}^{2,p}(\Omega)), T_{0}\leq T, \textrm{if}\,\,T<\infty\,\,\mathrm{and }\,\,T_{0}<T,\textrm{if}\,\,\,T=\infty
\end{equation}
\begin{equation}\label{reglplqlap2}
\frac{\partial\textbf{u}}{\partial t}\in L^{q}(0,T;\,\textbf{L}^{p}_{\sigma,\tau}(\Omega))
\end{equation}
and
\begin{equation}\label{estlplqlap}
\int_{0}^{T}\Big\Vert\frac{\partial\textbf{u}}{\partial t}\Big\Vert^{q}_{\textbf{L}^{p}(\Omega)}dt+\int_{0}^{T}\Vert\Delta\textbf{u}(t)\Vert^{q}_{\textbf{L}^{p}(\Omega)}dt\leq C(p,q,\Omega)\int_{0}^{T}\Vert\textbf{f}(t)\Vert^{q}_{\textbf{L}^{p}(\Omega)}dt.
\end{equation}
\end{theo}
\begin{proof}
Since the operator $-A_{p}$ generates a bounded analytic semi-group in $\textbf{L}^{p}_{\sigma,\tau}(\Omega)$,  and $\textbf{f}\in L^{q}(0,T;\,\textbf{L}^{p}_{\sigma,\tau}(\Omega))$, problem (\ref{StokesProblem})--(\ref{u0})  has a unique solution $\textbf{u}\in C(0,T;\,\textbf{L}^{p}_{\sigma,\tau}(\Omega))$. To prove the maximal $L^p$-$L^q$ regularity (\ref{reglplqlap1})-(\ref{estlplqlap}) we proceed as follows.

By (\ref{S2E10}) we may write $\textbf{f}$ in the form, $\textbf{f}=\textbf{f}_1+\textbf{f}_2$ where $\textbf{f}_{1}\in L^{q}(0,T;\,\textbf{X}_{p})$ and $\textbf{f}_{2}\in L^{q}(0,T;\,\textbf{K}_{\tau}(\Omega))$. Thus the solution $\textbf{u}$ to (\ref{StokesProblem})--(\ref{u0}) is such that $\textbf{u}=\textbf{u}_1+\textbf{u}_2$, where $\textbf{u}_1$ and $\textbf{u}_2$ satisfy

\begin{equation}
 \label{inhenspf1}
 \left\{
\begin{array}{cccc}
\frac{\partial\textbf{u}_{1}}{\partial t} - \Delta \textbf{u
}_{1}=\textbf{f}_{1},& 
\mathrm{div}\,\textbf{u}_{1}= 0 &\mathrm{in}&\Omega\times (0,T), \\
\textbf{u}_{1}\cdot\textbf{n}=0,& 
\mathrm{\textbf{curl}}\,\textbf{u}_{1}\times \textbf{n} = \textbf{0} &\mathrm{on} & \Gamma\times (0,T), \\
&\textbf{u}_{1}(0)=\textbf{0} &\mathrm{in}&
\Omega
\end{array}
\right.
\end{equation}
and

\begin{equation}
 \label{inhenspf2}
 \left\{
\begin{array}{cccc}
\frac{\partial\textbf{u}_{2}}{\partial t} - \Delta \textbf{u
}_{2}=\textbf{f}_{2},& 
\mathrm{div}\,\textbf{u}_{2}= 0 &\mathrm{in}&\Omega\times (0,T), \\
\textbf{u}_{2}\cdot\textbf{n}=0,& 
\mathrm{\textbf{curl}}\,\textbf{u}_{2}\times \textbf{n} = \textbf{0} &\mathrm{on} & \Gamma\times (0,T), \\
&\textbf{u}_{2}(0)=\textbf{0} &\mathrm{in}&
\Omega
\end{array}
\right.
\end{equation}
respectively.

By \cite[Theorem 1.2]{ARMA} we know that $\textbf{u}_{1}$ satisfies
\begin{equation}\label{reglplqlapu1}
\textbf{u}_{1}\in L^{q}(0,T_{0};\,D(A_{p}))\cap W^{1,q}(0,T;\,\textbf{L}^{p}_{\sigma,\tau}(\Omega))
\end{equation}
\begin{equation}\label{estlplqlapu1}
\int_{0}^{T}\!\!\Big\Vert\frac{\partial\textbf{u}_{1}}{\partial t}\Big\Vert^{q}_{\textbf{L}^{p}(\Omega)}\!dt+\int_{0}^{T}\!\!\Vert\Delta\textbf{u}_{1}(t)\Vert^{q}_{\textbf{L}^{p}(\Omega)}\!dt\leq C(p,q,\Omega)\int_{0}^{T}\!\!\Vert\textbf{f}_{1}(t)\Vert^{q}_{\textbf{L}^{p}(\Omega)}dt.
\end{equation}

Set $\textbf{z}_{2}=\mathbf{curl}\,\textbf{u}_{2}$. Then $\textbf{z}_{2}$ is a solution of the problem
\begin{equation}
\left\{
\begin{array}{cccc}
\frac{\partial\textbf{z}_{2}}{\partial t} - \Delta \textbf{z}_{2}= \textbf{0},& \mathrm{div}\,\textbf{z}_{2}=0& \textrm{in}&
\Omega\times (0,T), \\
&\textbf{z}_{2}\times \textbf{n} = \textbf{0},& \textrm{on} & \Gamma\times (0,T), \\
&\textbf{z}_{2}(0)= \textbf{0} &\textrm{in} &\Omega.
\end{array}
\right.
\end{equation}
Thus, using \cite[Theorem 4.1]{JEE} we deduce that $\mathbf{curl}\,\textbf{u}_{2}=\textbf{z}_{2}=\textbf{0}$ in $\Omega$. This means that $\textbf{u}_{2}\in\textbf{K}_{\tau}(\Omega)$ and then 
\begin{equation}
\forall t\geq0,\quad\frac{\partial\textbf{u}_{2}(t)}{\partial t}=\textbf{f}_{2}(t)\quad\mathrm{in}\,\,\Omega.
\end{equation}
As a result $\textbf{u}_{2}$ satisfies
\begin{equation}\label{reglplqlapu2}
\textbf{u}_{2}\in L^{q}(0,T_{0};\,D(A_{p}))\cap W^{1,q}(0,T;\,\textbf{L}^{p}_{\sigma,\tau}(\Omega))
\end{equation}
and
\begin{equation}\label{estlplqlapu2}
\int_{0}^{T}\Big\Vert\frac{\partial\textbf{u}_{2}}{\partial t}\Big\Vert^{q}_{\textbf{L}^{p}(\Omega)}dt=\,\int_{0}^{T}\Vert\textbf{f}_{2}(t)\Vert^{q}_{\textbf{L}^{p}(\Omega)}dt \leq C(p,q,\Omega)\int_{0}^{T}\Vert\textbf{f}(t)\Vert^{q}_{\textbf{L}^{p}(\Omega)}dt.
\end{equation}
Thus putting together (\ref{reglplqlapu1})-(\ref{estlplqlapu1}) and (\ref{reglplqlapu2})-(\ref{estlplqlapu2}) we deduce   our result.
\end{proof}

We now extend  the previous result  to the more general case where  the external force $\textbf{f}\in L^{q}(0,T;\,\textbf{L}^{p}(\Omega))$ is not necessarily divergence free. It is used that the pressure can be decoupled, using the weak Neumann Problem (\ref{wn.1}). 

\begin{theo}[Strong Solutions for the inhomogeneous Stokes Problem]\label{Exisinhnsplp}
Let $T\in (0, \infty]$, $1<p,q<\infty$, 
$\textbf{f}\in\textbf{L}^{q}(0,T;\textbf{L}^{p}(\Omega))$ and $\textbf{u}_{0}=0$. The Problem (\ref{StokesProblem})-(\ref{u0}) has a unique solution $(\textbf{u},\pi)$ such that
\begin{equation}\label{reglplqstokes1}
\textbf{u}\in L^{q}(0,T_{0};\,\textbf{W}^{2,p}(\Omega)),\,\,\,\,  T_{0}\leq T\,\,\,\textrm{if}\,\,\,T<\infty\,\,\,\, \mathrm{and }\,\,\,\,T_{0}<T\,\,\,\textrm{if}\,\,\,T=\infty,
\end{equation}
\begin{equation}\label{reglplqstokes2}
\pi\in L^{q}(0,T;\,W^{1,p}(\Omega)/\mathbb{R}),\qquad \frac{\partial\textbf{u}}{\partial t}\in L^{q}(0,T;\,\textbf{L}^{p}(\Omega))
\end{equation}
and
\begin{eqnarray}\label{estlplqstokes}
\int_{0}^{T}\Big\Vert\frac{\partial\textbf{u}}{\partial t}\Big\Vert^{q}_{\textbf{L}^{p}(\Omega)}dt+\,\int_{0}^{T}\Vert\Delta\textbf{u}(t)\Vert^{q}_{\textbf{L}^{p}(\Omega)}dt+\int_{0}^{T}\Vert\pi(t)\Vert^{q}_{W^{1,p}(\Omega)/\mathbb{R}}dt\\
\leq\,C(p,q,\Omega)\,\int_{0}^{T}\Vert\textbf{f}(t)\Vert^{q}_{\textbf{L}^{p}(\Omega)}dt.
\end{eqnarray}
\end{theo}

\begin{proof} 
As we saw in  Section \ref{section2} when defining the Helmholtz projection $P$,  for every
$\textbf{f}\in\textbf{L}^{q}(0,T;\,\textbf{L}^{p}(\Omega))$, and almost every $0<t<T,$ the problem
\begin{equation}
\mathrm{div}(\mathrm{\textbf{grad}}\,\pi(t)-\textbf{f}(t))=0\,\,\textrm{in}\,\Omega,\,\,\,\,
(\mathrm{\textbf{grad}}\,\pi(t)-\textbf{f}(t))\cdot\textbf{n}=0\,\,\textrm{on}\,\Gamma,
\end{equation}
has a unique solution $\pi(t)\in W^{1,p}(\Omega)/\mathbb{R}$ that
satisfies the estimate
\begin{equation}
\textrm{for}\,\,\textrm{a.e.}\,\,t\in(0,T)\qquad\|\pi(t)\|_{W^{1,p}(\Omega)/\mathbb{R}}\,\leq\,C(\Omega)\|\textbf{f}(t)\|_{\textbf{L}^{p}(\Omega)}.
\end{equation}
It follows  that: 
$$\pi\in L^{q}(0,T;\,W^{1,p}(\Omega)/\mathbb{R}),\,\,\hbox{and}\,\,\,(\textbf{f}-\mathrm{\textbf{grad}}\,\pi)\in
L^{q}(0,T;\,\textbf{L}^{p}_{\sigma,\tau}(\Omega)).$$
 As a result, thanks to Theorem \ref{existinhensp},  Problem (\ref{StokesProblem})-(\ref{u0})  has a unique solution $(\textbf{u},\pi)$ satisfying (\ref{reglplqstokes1})-(\ref{estlplqstokes}).
\end{proof}
Similar results  hold for weak and  very weak solutions. 
\begin{theo}[Weak Solutions for the inhomogeneous Stokes Problem]\label{Existinhsphdiv}
Let $1<p,q<\infty$, 
$\textbf{u}_{0}=0$ and let $\textbf{f}\in L^{q}(0,T;\,[\textbf{H}^{p'}_{0}(\mathrm{div},\Omega)]')$, $0<T\leq\infty$. The Problem (\ref{StokesProblem})-(\ref{u0}) has a unique solution $(\textbf{u},\pi)$ satisfying
\begin{equation}\label{weakmaxregl1}
\textbf{u}\in L^{q}(0,T_{0};\,\,\textbf{W}^{1,p}(\Omega)),\,\,\,\,  T_{0}\leq T\,\,\,\textrm{if}\,\,\,T<\infty\,\,\,\, \mathrm{and }\,\,\,\,T_{0}<T\,\,\,\textrm{if}\,\,\,T=\infty,
\end{equation}
\begin{equation}\label{weakmaxregl2}
\pi\in L^{q}(0,T;\,\,L^{p}(\Omega)/\mathbb{R}),\qquad\frac{\partial\textbf{u}}{\partial t}\in L^{q}(0,T;\,\in[\textbf{H}^{p'}_{0}(\mathrm{div}\Omega)]'_{\sigma,\tau})
\end{equation}
and
\begin{eqnarray}\label{weakmaxregl3}
&&\int_{0}^{T}\Big\Vert\frac{\partial\textbf{u}}{\partial t}\Big\Vert^{q}_{[\textbf{H}^{p'}_{0}(\mathrm{div}\Omega)]'}dt+\,\int_{0}^{T}\Vert\Delta\textbf{u}(t)\Vert^{q}_{[\textbf{H}^{p'}_{0}(\mathrm{div}\Omega)]'}dt+\nonumber \\
&&\hskip 0.15cm +\,\int_{0}^{T}\Vert\pi(t)\Vert^{q}_{L^{p}(\Omega)/\mathbb{R}}dt
\leq\,C(p,q,\Omega)\,\int_{0}^{T}\Vert\textbf{f}(t)\Vert^{q}_{[\textbf{H}^{p'}_{0}(\mathrm{div}\Omega)]'}dt.
\end{eqnarray}
 \end{theo}
 \begin{proof}
 Suppose first $\textbf{f}\in L^{q}(0,T;\,[\textbf{H}^{p'}_{0}(\mathrm{div},\Omega)]' _{ \sigma , \tau  })$. Using that $-B_p$ generates a bounded analytic semigroup in $[\textbf{H}^{p'}_{0}(\mathrm{div},\Omega)]' _{ \sigma , \tau  }$ we deduce the existence of a unique weak solution 
$ \textbf{u}\in C(0, T; [\textbf{H}^{p'}_{0}(\mathrm{div},\Omega)]' _{ \sigma , \tau  })$ of (\ref{StokesProblem})--(\ref{u0}).  

By (\ref{S2E11}) we may write now $\textbf{f}$ as, $\textbf{f}=\textbf{f}_1+\textbf{f}_2$ where $\textbf{f}_{1}\in L^{q}(0,T;\,\textbf{Y}_{p})$ and $\textbf{f}_{2}\in L^{q}(0,T;\,\textbf{K}_{\tau}(\Omega))$. Proceeding as in the proof of Theorem \ref{existinhensp} we deduce that the solution $\textbf{u}$ to problem (\ref{StokesProblem})--(\ref{u0}) is such that $\textbf{u}=\textbf{u}_1+\textbf{u}_2$, where $\textbf{u}_1$ and $\textbf{u}_2$ are weak solutions of (\ref{inhenspf1}) and (\ref{inhenspf2}) respectively and that $\textbf{u}_2\in\textbf{K}_{\tau}(\Omega)$ for almost all $0<t\leq T$. Using \cite[Proposition 6.4, Remark 7.15]{ARMA} we deduce that the solution $\textbf{u}$ satisfies the maximal regularity (\ref{weakmaxregl1})-(\ref{weakmaxregl3}).
Suppose now  $\textbf{f}\in L^{q}(0,T;\,[\textbf{H}^{p'}_{0}(\mathrm{div},\Omega)]')$. Then, for almost every $t\in (0, T)$, there exists  a unique solution $\pi(t)\in L^{p}(\Omega)/\mathbb{R}$ such that:
 \begin{equation}
 \Vert\pi\Vert_{L^{p}(\Omega)/\mathbb{R}}\,\leq\,C_{2}(\Omega,p)\Vert\textbf{f}\Vert_{[\textbf{H}^{p'}_{0}(\mathrm{div},\Omega)]'},
 \end{equation}
(cf. \cite{AMN}), and then also:
$$\pi\in L^{q}(0,T;\,L^p(\Omega)/\mathbb{R})\,\,\,\,\hbox{and}\,\,\,\,
(\textbf{f}-\mathrm{\textbf{grad}}\,\pi)\in
L^{q}(0,T;\,[\textbf{H}^{p'}_{0}(\mathrm{div},\Omega)]' _{ \sigma , \tau  }).$$
We deduce from the previous step that $(\textbf{u}, \pi )$ satisfies (\ref{weakmaxregl1})-(\ref{weakmaxregl3}). \end{proof}

\begin{theo}[Very weak solutions for the inhomogeneous Stokes Problem]\label{Existinhsptp}\hfill \break
Let $T\in (0, \infty]$, $1<p,q<\infty$,  $\textbf{u}_{0}=0$ and $\textbf{f}\in L^{q}(0,T;\,[\textbf{T}^{p'}(\Omega)]')$. Then the time dependent Stokes Problem (\ref{StokesProblem})-(\ref{u0}) has a unique solution $(\textbf{u},\pi)$ satisfying
\begin{equation}
\label{S3E1}
\textbf{u}\in L^{q}(0,T_{0};\,\textbf{L}^{p}(\Omega)),\,\,\,\,  T_{0}\leq T\,\,\,\textrm{if}\,\,\,T<\infty\,\,\,\, \mathrm{and }\,\,\,\,T_{0}<T\,\,\,\textrm{if}\,\,\,T=\infty,
\end{equation}  
\begin{equation}
\label{S3E2}
\pi\in L^{q}(0,T;\,\,W^{-1,p}(\Omega)/\mathbb{R}),\qquad\frac{\partial\textbf{u}}{\partial t}\in L^{q}(0,T;\,\in[\textbf{T}^{p'}(\Omega)]'_{\sigma,\tau})
\end{equation}
and
\begin{eqnarray}
\label{S3E3}
&&\int_{0}^{T}\Big\Vert\frac{\partial\textbf{u}}{\partial t}\Big\Vert^{q}_{[\textbf{T}^{p'}(\Omega)]'}dt+\,\int_{0}^{T}\Vert\Delta\textbf{u}(t)\Vert^{q}_{[\textbf{T}^{p'}(\Omega)]'}dt+
\nonumber \\
&&\hskip 0.15cm +\int_{0}^{T}\Vert\pi(t)\Vert^{q}_{W^{-1,p}(\Omega)/\mathbb{R}}dt \leq\,C(p,q,\Omega)\,\int_{0}^{T}\Vert\textbf{f}(t)\Vert^{q}_{[\textbf{T}^{p'}(\Omega)]'}dt.
\end{eqnarray}
 \end{theo}
\begin{proof}
The proof follows the same arguments as those in the proof of Theorem \ref{Existinhsphdiv}. In a first step one uses  $C_p$, the  analytic semigroup on $[\textbf{T}^{p'}(\Omega)]'_{\sigma, \tau}$ and (\ref{S2E14}) to prove  that (\ref{S3E1})-(\ref{S3E3}) are satisfied when $f\in [\textbf{T}^{p'}(\Omega)]'_{\sigma, \tau}$. In the general case, one  uses the results in  \cite{AMN} to obtain $\pi\in L^{q}(0,T;\,\,W^{-1,p}(\Omega)/\mathbb{R}) $ such that $(\textbf{f}-\mathrm{\textbf{grad}}\,\pi)\in
L^{q}(0,T;\,[\textbf{T}^{p'}(\Omega)]'_{\sigma, \tau})$, and the result follows using the first step.
\end{proof}

\noindent
\textbf{Acknowledgements} M.E. is supported by Grants MTM2014-52347-C2-1-R of DGES  and  IT641-13 of Basque Government. H. Al B. acknowledges the support of the GA\v CR (Czech Science Foundation) project  16-03230S  in the framework of RVO: 67985840.

\end{document}